\documentclass[11pt]{amsart}
\usepackage{xcolor}
\usepackage{hyperref}
\hypersetup{
    colorlinks = false,
    linkbordercolor = {white},
    citebordercolor = {white}
    }
\usepackage[capitalise,noabbrev]{cleveref}
\usepackage[all]{xy}

\usepackage{amsmath,amscd}
\usepackage{amssymb}
\usepackage{mathtools}
\usepackage{stmaryrd}
\usepackage{url}
\usepackage{tikz-cd}
\usepackage{enumitem}
\usepackage{fullpage}
\usepackage{musicography}

\usepackage{physics}

\usepackage{tikz}

\usepackage{enumitem,kantlipsum}

\author{Michael K. Brown}

\author{Andrew J. Soto Levins}

\author{Prashanth Sridhar}

\newcommand{\Addresses}{{
	\vskip\baselineskip
  	\footnotesize
  	\noindent \textsc{Department of Mathematics and Statistics, Auburn University} \par\nopagebreak
	\noindent \textit{E-mail addresses:} \texttt{mkb0096@auburn.edu, pzs0094@auburn.edu}
    \\
    \\
  \noindent \textsc{Department of Mathematics and Statistics, Texas Tech University} \par\nopagebreak
	\noindent \textit{E-mail address:} \texttt{ansotole@ttu.edu}
    
 }}

\numberwithin{equation}{section}
\newtheorem{lemma}[equation]{Lemma}
\newtheorem{lem}[equation]{Lemma}

\newtheorem{prop}[equation]{Proposition}
\newtheorem{cor}[equation]{Corollary}

\newtheorem{claim*}{Claim}
\newtheorem{thm}[equation]{Theorem}

\theoremstyle{definition}
\newtheorem{defn}[equation]{Definition}
\newtheorem{dfn}[equation]{Definition}
\newtheorem{notation}[equation]{Notation}
\newtheorem{example}[equation]{Example}

\newtheorem{setup}[equation]{Setup}

\theoremstyle{remark}
\newtheorem{remark}[equation]{Remark}

\newcommand{\mfrak}[1]{\mathfrak{#1}}

\renewcommand{\k}{\Bbbk}

\renewcommand{\a}{\alpha}

\newcommand{\m}{\mfrak{m}}
\newcommand{\n}{\mfrak{n}}

\newcommand{\injdim}{\operatorname{inj\,dim}}

\newcommand{\amp}{\operatorname{amp}}
\newcommand{\Hom}{\operatorname{Hom}}
\newcommand{\RHom}{\text{{\bf R}$\Hom$}}
\newcommand{\RuHom}{\text{{\bf R}\underline{$\Hom$}}}

\newcommand{\D}{\msf{D}}
\newcommand{\kk}{\mathbf{k}}

\newcommand{\del}{\partial}
\newcommand{\Mod}{\operatorname{Mod}}

\newcommand{\dGamma}{\mathbf{R}\Gamma}

\def\nc{\newcommand}
\nc{\on}{\operatorname}
\nc{\bideg}{\on{bideg}}
\nc{\xra}{\xrightarrow}
\def\phi{\varphi}
\def\th{\on{th}}

\def\D{\on{D}}
\def\Db{\D^{\on{b}}}

\nc{\into}{\hookrightarrow}
\nc{\onto}{\twoheadrightarrow}
\nc{\LL}{\mathbf{L}}
\nc{\RR}{\mathbf{R}}
\nc{\Perf}{\on{Perf}}
\nc{\nat}{\natural}
\nc{\tors}{\on{tors}}
\nc{\Tors}{\on{Tors}}

\def\Mod{\on{Mod}}
\nc{\qgr}{\on{qgr}}
\nc{\Qgr}{\on{Qgr}}
\nc{\fQgr}{\on{Qgr}^{\on{f}}}
\nc{\colim}{\on{colim}}
\def\Z{\mathbb{Z}}
\nc{\Ext}{\on{Ext}}

\nc{\om}{\omega}
\nc{\w}{\widetilde}
\nc{\PP}{\mathbb{P}}
\nc{\mf}{\on{mf}}
\nc{\OO}{\mathcal{O}}
\nc{\Proj}{\on{Proj}}
\nc{\Qcoh}{\on{Qcoh}}
\nc{\coh}{\on{coh}}
\nc{\Tor}{\on{Tor}}
\nc{\Modf}{\Mod^{\on{f}}}
\def\op{\on{op}}
\nc{\ce}{\coloneqq}

\def\k{\kk}
\nc{\Com}{\on{Com}}
\nc{\A}{\mathcal{A}}
\nc{\B}{\mathcal{B}}
\nc{\C}{\mathcal{C}}
\nc{\Sh}{\on{Sh}}
\nc{\QCoh}{\on{QCoh}}
\nc{\Coh}{\on{Coh}}
\nc{\fQCoh}{\QCoh^{\on{f}}}
\nc{\ov}{\overline}
\nc{\End}{\on{\underline{End}}}
\def\MR#1{}

\nc{\Qgrf}{\Qgr^{\on{f}}}
\nc{\uHom}{\underline{\Hom}}
\nc{\Inj}{\mathrm{Inj}}
\nc{\proj}{\mathrm{Proj}}
\nc{\spec}{\mathrm{Spec}}
\nc{\uExt}{\underline{\Ext}}
\nc{\co}{\colon}

\nc{\lcd}{\on{lcd}}

\def\A{\mathcal{A}}

\usepackage{comment}

\begin{document}
\title{Existence of balanced dualizing dg-modules}

\begin{abstract}
We describe cohomological conditions that are necessary and sufficient for the existence of balanced dualizing dg-modules, generalizing a theorem of Van den Bergh for balanced dualizing complexes over graded algebras. As a consequence, we show that a dg-algebra satisfying certain finiteness conditions admits a balanced dualizing dg-module if and only if its zeroth cohomology algebra admits a balanced dualizing complex. Additionally, we obtain a host of new examples of dg-algebras whose associated noncommutative spaces satisfy Serre duality.
\end{abstract}


\numberwithin{equation}{section}

\maketitle
\setcounter{tocdepth}{1}

\setcounter{section}{0}

\section{Introduction}

Let $\k$ be a field and $S = \bigoplus_{i \ge 0} S_i$ a Noetherian graded $\k$-algebra such that $S_0 = \k$. When $S$ is commutative, it determines a projective variety $X = \on{Proj}(S)$. In general, no such projective variety exists, but the fundamental insight of noncommutative projective geometry, originating in work of Artin-Zhang~\cite{AZ}, is that one may associate to $S$ a triangulated category $\D_{\qgr}(S)$ with many of the same homological features as the bounded derived category $\Db(X)$. More precisely, the category $\D_{\qgr}(S)$ is defined to be the quotient of $\Db(S)$ by complexes $C$ such that $\dim_\k H(C) < \infty$; when $S$ is commutative and generated in degree one, $\D_{\qgr}(S)$ is equivalent to $\Db(X)$. 

For example, Yekutieli-Zhang~\cite{Yekutieli_Zhang} have proven that a generalization of Serre duality holds for $\D_{\qgr}(S)$ when $S$ admits a balanced dualizing complex, a notion due to Yekutieli~\cite{yek92}. We recall that a dualizing complex for $S$ is, roughly speaking, a complex $R$ such that the functor $\RuHom_S( -, R)$ determines an equivalence $\Db(S) \xra{\simeq} \Db(S)^{\op}$, and a \emph{balanced} dualizing complex has an additional rigidifying property that forces it to be unique up to isomorphism. Van den Bergh subsequently characterized the algebras admitting balanced dualizing complexes~\cite[Theorem 6.3]{VdB_dualizing}, thus determining exactly when $\D_{\qgr}(S)$ enjoys Serre duality. We refer the reader to Yekutieli's textbook~\cite{yekutielibook} for a comprehensive discussion of balanced dualizing complexes and the important role they play in noncommutative algebraic geometry.

Suppose now that $A$ is a differential bigraded $\k$-algebra, i.e. a dg-algebra with both a cohomological and internal grading (see Section~\ref{sec:existence} for the precise definition, and see Notation~\ref{notation} for our indexing conventions). Assume $A$ satisfies the conditions in Setup~\ref{setup}. One may define the category $\D_{\qgr}(A)$ just as in the case of algebras above (see Section~\ref{sec:serre} for details), and the first- and third-named authors have recently begun a study of the geometric properties of the noncommutative space $\D_{\qgr}(A)$~\cite{serreduality_dgalgebras,BROWN2025110035}, i.e. a study of \emph{derived} noncommutative geometry. In particular, they introduce in~\cite{serreduality_dgalgebras} the notion of a balanced dualizing dg-module for $A$, which is the natural differential graded generalization of a balanced dualizing complex. Moreover, they prove that $\D_{\qgr}(A)$ satisfies Serre duality when~$A$ admits a balanced dualizing dg-module~\cite[Theorem 1.3]{serreduality_dgalgebras}, generalizing Yekutieli-Zhang's noncommutative Serre duality theorem. This motivates the question: under what conditions does~$A$ admit a balanced dualizing dg-module? Our main result answers this question, generalizing the aforementioned theorem of Van den Bergh~\cite[Theorem 6.3]{VdB_dualizing}: 

\begin{thm}
\label{thm:intro}
A differential bigraded $\k$-algebra $A$ as in Setup~\ref{setup} admits a balanced dualizing dg-module if and only if the graded $\k$-algebra $H^0(A)$ and its opposite algebra $H^0(A)^{\op}$ satisfy Artin-Zhang's condition~$\chi$ \cite{AZ} and have finite local cohomological dimension (Definition~\ref{def:lcd_classical}).
\end{thm}

See Theorem~\ref{thm:vdb} for a more precise version of Theorem~\ref{thm:intro}. It follows that Serre duality holds for $\D_{\qgr}(A)$ when the equivalent conditions in Theorem~\ref{thm:intro} are satisfied: see Corollary~\ref{cor:serre}. As an additional consequence, we prove the following (see \Cref{cor:A0} for a more precise statement):

\begin{cor}
\label{cor:intro}
A differential bigraded $\k$-algebra $A$ as in Setup~\ref{setup} admits a balanced dualizing dg-module if and only if $H^0(A)$ admits a balanced dualizing complex.
\end{cor}

Corollary~\ref{cor:intro} provides an ample source of examples of dg-algebras with balanced dualizing dg-modules: see Example~\ref{ex:explicit}. Theorem~\ref{thm:intro} and Corollary~\ref{cor:intro} underscore how cleanly the theory of balanced dualizing complexes extends to the differential graded setting, building on the evidence already provided by \cite{serreduality_dgalgebras,BROWN2025110035} that the theory of noncommutative projective geometry generalizes quite robustly to dg-algebras. 

\medskip
As an application of Theorem~\ref{thm:intro}, we show that, under certain conditions, the category $\Perf_{\qgr}(A)$ of perfect complexes in $\D_{\qgr}(A)$ (which we define in Section~\ref{sec:serre}) admits a Serre functor:

\begin{cor}
\label{cor:serrefunctor}
Suppose $A$ is Gorenstein (Definition~\ref{def:gorenstein}) and that $A$ has finite injective dimension as a right dg-$A$-module, in the sense of \cite[Definition 2.1]{finitistic_dimensions}. Choose $a, n \in \Z$ such that we have $\RuHom_A(\k, A) \cong \k(a)[-n]$ in $\D(A)$. If $H^0(A)$ admits a balanced dualizing complex, then the functor $S \co \Perf_{\qgr}(A) \to \Perf_{\qgr}(A)$ given by $S(\mathcal{F}) = \mathcal{F}(-a)[n]$ is a Serre functor.
\end{cor}

\subsection*{Overview of the paper} Section~\ref{sec:derived} contains background on derived categories of differential bigraded algebras. Section~\ref{sec:localcoh} is devoted to establishing a number of technical results on local cohomology of differential bigraded modules; in particular, we provide a detailed study of condition~$\chi$ for dg-algebras. 
In Section~\ref{sec:balanced}, we prove our main result, Theorem~\ref{thm:vdb}, and we discuss several consequences and examples. In Section~\ref{sec:serre}, we discuss the implications of our results for Serre duality over $\D_{\qgr}(A)$; in particular, we prove Corollary~\ref{cor:serrefunctor}. 







\begin{notation}
\label{notation}
Throughout, $\k$ denotes a field. We will consider bigraded $\kk$-vector spaces $V = \bigoplus_{i, j \in \Z} V_i^j$, where the superscript denotes a cohomological grading, and the subscript refers to an internal grading. 
Given $v \in V^j_i$, we write $\deg(v) = i$ and $|v| = j$. 
We will denote the $m^{\th}$ shift of $V$ in internal (resp. cohomological) degree by $V(m)$ (resp. $V[m]$). That is, $V(m)_i^j = V_{i + m}^j$, and $V[m]_i^j = V_i^{j + m}$. Given a complex $C$, its \emph{amplitude} is
$$
\amp(C) \ce \sup\{j - \ell \text{ : } H^j(C) \ne 0 \text{ and } H^\ell(C) \ne 0\}.
$$
We define $\inf(C) \ce \inf\{j \text{ : } H^j(C) \ne 0\}$, and $\sup(C)$ is defined similarly. 
\end{notation}

\subsection*{Acknowledgments} The first author was partially supported by NSF grant DMS-2302373.
\section{Existence of balanced dualizing dg-modules}
\label{sec:existence}

Let $A$ be a \emph{differential bigraded $\kk$-algebra}, i.e., a $\k$-algebra equipped with a bigrading $A = \bigoplus_{(i, j) \in \Z^2} A_i^j$ and a degree $(0, 1)$
$\kk$-linear map $\del_A$ that squares to 0 and satisfies the Leibniz rule:
$
\del_A(xy) = \del_A(x)y + (-1)^{|x|}x\del_A(y).
$
We refer to \cite[Definition 2.2]{BROWN2025110035} for the definitions of left and right differential bigraded $A$-modules. Throughout the paper, we abbreviate ``differential bigraded $\k$-algebra" (resp. ``differential bigraded $A$-module")  to ``dg-algebra" (resp. ``dg-$A$-module"). All dg-$A$-modules are assumed to be right modules unless otherwise noted. 
Let $A^{\op}$ denote the \emph{opposite dg-algebra of $A$} \cite[Definition 2.1]{BROWN2025110035}. If $A$ and $B$ are differential bigraded $\k$-algebras, a  \emph{dg-$A$-$B$-bimodule} is a right $A^{\op} \otimes_k B$ module. 
Given a differential bigraded $A$-module~$M$, we set $M_i \coloneqq \bigoplus_{j \in \Z} M_i^j$ and $M^j \coloneqq \bigoplus_{i \in \Z} M^j_i$. We say $A$ is \emph{connected} if 
$A_0 = A_0^0 = \k$,
and $A_i^j = 0$ when $i<0$ or $j > 0$. If $A$ is connected, then $A_0 = \k$ is a dg-$A$-module.

\medskip
In this section, we work under the following setup, which is the same as~\cite[Setup 2.8]{BROWN2025110035}.

\begin{setup}
\label{setup}
Let $A$ be a connected differential bigraded $\k$-algebra such that $H^0(A)$ is a Noetherian ring, and the total cohomology algebra $H(A)$ is finitely generated as an $H^0(A)$-module. 
\end{setup}

\subsection{Derived categories of differential bigraded algebras}
\label{sec:derived}

Let $\D(A)$ (resp. $\Db(A)$) denote the (resp. bounded) derived category of $A$, as defined in \cite[Section 2.2]{BROWN2025110035}. The objects of $\D(A)$ are  dg-$A$-modules, and the objects of $\Db(A)$ are dg-$A$-modules $M$ such that $H(M)$ is finitely generated over $H(A)$. 
Restriction of scalars along the dg-algebra morphism $A \to H^0(A)$ induces a functor $\Db(H^0(A)) \to \Db(A)$.

\begin{prop}
\label{prop:gen}
The category $\Db(A)$ is classically generated by finitely generated $H^0(A)$-modules. 
\end{prop}

\begin{proof}
Let $M \in \Db(A)$. We argue by induction on $\amp(M)$. If $\amp(M) = 0$, then $M$ is (a cohomological shift of) a finitely generated $H^0(A)$-module. Suppose $\amp(M) > 0$, let $m \ce \inf(M)$, and let $\sigma^{\le m}$ denote the smart truncation functor, as defined in \cite[Remark 2.7]{BROWN2025110035}. The short exact sequence $0 \to \sigma^{\le m}M \to M \to M / \sigma^{\le m}M \to 0$ and induction imply the result. 
\end{proof}

Let $M$ (resp. $N$) be a right (resp. left) dg-$A$-module. We define the tensor product $M \otimes_A N$ and the derived tensor product $M \otimes_A^\LL N$  as in \cite[Section 2]{BROWN2025110035}. Similarly, if $M$ and $N$ are right dg-$A$-modules, then we define the Hom space $\Hom_A(M, N)$ and internal Hom dg-$A$-module $\uHom_A(M, N)$, as well as their derived variants $\RHom_A(M, N)$ and $\RuHom_A(M, N)$, as in \cite[Section 2]{BROWN2025110035}. If $M$ (resp. $N$) is a dg-$A$-$A$-bimodule, then $\uHom_A(M, N)$ and $\RuHom_A(M, N)$ are right (resp. left) dg-$A$-modules. We write $\Ext^j_A(M, N) \ce H^j\RuHom_A(M, N)_0$ and $\uExt^j_A(M, N) \ce H^j\RuHom_A(M, N)$. 

\begin{prop}
\label{prop:extfinite}
Let $M, N \in \Db(A)$. We have:
\begin{enumerate}
\item $\dim_\k \uExt_A^*(M, N)_i < \infty$ for all $i \in \Z$. 
\item $\uExt_A^*(M, N)_i = 0$ for $i \ll 0$. 
\end{enumerate}
\end{prop}

\begin{proof}
Both statements follow by computing $\uExt_A(M,N)$ using a semifree resolution of $M$ as constructed in the proof of \cite[Proposition 2.16]{BROWN2025110035}.
\end{proof}

Given a dg-$A$-module $V$, we set $V^* \ce\uHom_\k(V, \k)$. The functor $\D(A) \to \D(A^{\op})^{\op}$ given by $M \mapsto M^*$ interchanges $K$-projective and $K$-injective dg-modules (see \cite[Section 2.1]{BROWN2025110035} for background on these objects). Let $\D^{\mathrm{lf}}(A)$ denote the full subcategory of $\D(A)$ given by objects $M$ that are locally finite, i.e. such that $\dim_\k H^j(M)_i < \infty$ for all $i, j \in \Z$. The following result is standard and follows from the existence of $K$-projective and $K$-injective resolutions.


\begin{lem}
\label{lem:matlis}
    The functor $\D^{\mathrm{lf}}(A) \to \D^{\mathrm{lf}}(A^{\op})^{\op}$ given by $M \mapsto M^*$ is an equivalence. Moreover, for all $M,N\in \D^{\mathrm{lf}}(A)$, we have a quasi-isomorphism $\RuHom_A(M, N)\simeq \RuHom_{A^{\op}}(N^*,M^*)$.
\end{lem}

We recall what it means for $A$ to be Gorenstein:

\begin{defn}\label{def:gorenstein}
We say the differential bigraded $\k$-algebra $A$ is \emph{Gorenstein} if:
\begin{enumerate}
\item The functors $\RuHom_A( - , A)$ and $\RuHom_{A^{\op}}( - , A)$ map $\Db(A)$ to $\Db(A^{\op})$ and $\Db(A^{\op})$ to $\Db(A)$, respectively. 
\item Given $M \in \Db(A)$ and $N \in \Db(A^{\op})$, the canonical maps
$$
M \to \RuHom_{A^{\op}}(\RuHom_A(M, A), A) \quad \text{and} \quad N \to \RuHom_{A}(\RuHom_{A^{\op}}(N, A), A)
$$
are isomorphisms in $\D(A)$ and $\D(A^{\op})$, respectively.
\item There is an isomorphism $\RuHom_A(\k, A) \cong \k(a)[-n]$ in $\D(A)$ for some $a, n \in \Z$. 
\end{enumerate}
\end{defn}

There are additional definitions of Gorenstein dg-algebras in the literature: see \cite[Remark 2.19]{BROWN2025110035} for a detailed discussion. We refer to \cite[Examples 2.23---2.27]{BROWN2025110035} for several examples of Gorenstein dg-algebras.

\subsection{Local cohomology of differential bigraded modules}
\label{sec:localcoh}
Let $\m$ denote the ideal $A_{\ge 1}$ in $A$. Given a dg-$A$-module $M$, its \emph{derived $\m$-torsion} is defined to be
$$
\RR\Gamma_{\m}(M) \ce \underset{d \to \infty}{\colim}\text{ } \RuHom_A(A / A_{\ge d}, M),
$$
and its $i^{\th}$ \emph{local cohomology} is $H^i_\m(M) \ce H^i \RR\Gamma_\m(M)$ \cite[Section 3.1]{serreduality_dgalgebras}. In this subsection, we establish several technical results about derived $\m$-torsion of dg-$A$-modules. 
The first is an extension of a result of Van den Bergh~\cite[Lemma 4.3]{VdB_dualizing}; before stating it, we recall the definition of local cohomological dimension for graded algebras:

\begin{defn}
\label{def:lcd_classical}
Let $S$ be a connected $\k$-algebra. The \emph{local cohomological dimension of $S$} is
$$
\lcd(S) \ce \inf\{d \text{ : }  H^i_\m(M)=0 \text{ for all $i> d$ and all graded $S$-modules $M$}\}.
$$
\end{defn}

\begin{lemma}\label{lem:colimits}
   If $\lcd(H^0(A)) < \infty$, then  $\RR\Gamma_\m \co \D(A) \to \D(A)$ commutes with filtered colimits.
\end{lemma}

\begin{proof}
   Let $\{M_{i}\}_{i\in I}$ be a filtered system of objects in $\D(A)$. Our use of subscripts here collides with our notation for internally graded components, but since we do not refer to the internal grading on the $M_i$'s in this proof, this should not cause confusion. Applying appropriate cohomological shifts, it suffices to show that the canonical map
   \begin{equation}
   \label{eqn:colim}
   \underset{i\in I}{\colim} \text{ }H^0_\m(M_i)\to H^0_\m(\underset{i\in I}{\colim} \text{ }M_i)
   \end{equation}
   is an isomorphism. Let us first assume that there exists $t \in \Z$ such that $(M_i)^j = 0$ for all~$j < t$ and all $i \in \Z$. Fix $n \ge 0$, and let $F$ be a minimal semifree resolution of $A / A_{\ge n}$ as constructed in \cite[Proposition 2.16]{BROWN2025110035}. It follows from this construction that $F^j = 0$ for $j > 0$, and $F^{\ge j}$ is a perfect dg-$A$-module for all $j \le 0$. Fix $s < t$. For any dg-$A$-module $N$ satisfying $N^j = 0$ for $j < t$, the natural  map 
   $H^0 \uHom_A(F, N) \to H^0 \uHom_A(F^{\ge s}, N)$ is an isomorphism. It follows that there is a commutative square
   $$
    \xymatrix{
    \underset{i\in I}{\colim} \text{ }H^0 \uHom_A(F, M_i) \ar[r] \ar[d]^{\cong} & H^0 \uHom_A(F, \underset{i\in I}{\colim} \text{ }M_i) \ar[d]^{\cong}  \\
    \underset{i\in I}{\colim} \text{ }H^0 \uHom_A(F^{\ge s}, M_i) \ar[r] & H^0 \uHom_A(F^{\ge s}, \underset{i\in I}{\colim} \text{ }M_i).
    }
   $$
    Since $F^{\ge s}$ is a compact object in $\D(A)$, the bottom map in this square is an isomorphism, and hence so is the top. It follows that \eqref{eqn:colim} is an isomorphism in this case.
   
    Given $t \in \Z$, let $\sigma^{\geq t}$ denote the smart truncation functor described in~\cite[Remark 2.7]{BROWN2025110035}. We have $\underset{i\in I}{\colim}\text{ }(\sigma^{\geq t}M_i) = \sigma^{\geq t} (\underset{i\in I}{\colim} \text{ } M_i)$. Set $d \ce \lcd(H^0(A))$, and let $N \in \D(A)$. By the argument above, it suffices to show that the natural map $H^0_\m(N) \to H^0_\m(\sigma^{\geq -d-1}N)$ is an isomorphism.  

By \cite[Proposition 3.7]{serreduality_dgalgebras}, we may assume $A=A^0$ . Let $\RR\underset{t}{\lim}(\sigma^{\ge -t} N) \in \D(A)$ denote the derived limit of the system $\{\sigma^{\geq -t}N\}_{t\geq 0}$. Since the natural map $N \to \RR\underset{t}{\lim}(\sigma^{\ge -t} N)$ is an isomorphism in $\D(A)$, and the functor $\RR\Gamma_\m$ commutes with derived limits, there is a natural isomorphism $\RR\Gamma_\m(N) \xra{\cong} \RR\underset{t}{\lim}(\RR\Gamma_\m(\sigma^{\ge -t} N))$ in $\D(A)$. We therefore have a canonical short exact sequence


\begin{equation}
\label{eqn:limit}
  0\to \RR\underset{t}{\lim}^1(H_\m^{-1}(\sigma^{\geq -t}N)) \to H^0_\m(N) \to \underset{t}{\lim} (H_{\m}^0(\sigma^{\geq -t}N)) \to 0
\end{equation}
(see e.g. \cite[Tag 07KY]{stacks-project}). For each $t\in \mathbb{Z}$, we have a triangle

\begin{equation}
  H^{-t}(N)[t] \to \sigma^{\geq -t}N \to \sigma^{\geq -t+1} N \to H^{-t}(N)[t+1].
\end{equation}
It follows that, for $t > d$, the canonical map $H_{\m}^0(\sigma^{\geq -t}N) \to H_{\m}^0(\sigma^{\geq -t+1}N)$ is an isomorphism. The system $\{H_\m^0(\sigma^{\geq -t}N)\}_{t \ge 0}$ thus satisfies the Mittag-Leffler condition~\cite[Definition 3.5.6]{Weibel}. By \cite[Corollary 3.5.4 and Proposition 3.5.7]{Weibel}, we conclude that $\RR\underset{t}{\lim}^1(H_\m^{-1}(\sigma^{\geq -t}N))=0$, and so the short exact sequence~\eqref{eqn:limit} implies that the canonical map $H^0_\m(N) \to \underset{t}{\lim} (H_{\m}^0(\sigma^{\geq -t}N))$ is an isomorphism. Moreover, we have shown that the system $\{H^0_{\m}(\sigma^{\geq -t}N)\}_{t \ge 0}$ stabilizes when $t > d$. This gives the desired isomorphism $H^0_{\m}(N)\cong H^0_{\m}(\sigma^{\geq -d - 1}N)$.
\end{proof}

Let $\D_{\Tors}(A)$ denote the subcategory of $\D(A)$ given by dg-$A$-modules $M$ with torsion cohomology, meaning that every $x \in H(M)$ satisfies $x \cdot H(A)_{\ge d} = 0$ for some $d \gg 0$. The functor~$\RR\Gamma_\m$ is the right adjoint of the fully faithful embedding $\D_{\Tors}(A) \into \D(A)$~\cite[Proposition 3.1]{serreduality_dgalgebras}. 

\begin{lemma}
\label{lem:counit}
Let $M \in \D(A)$. The counit map $\varepsilon \co \RR\Gamma_\m(M) \to M$ is an isomorphism in $\D(A)$ if and only if $M \in \D_{\Tors}(A)$. 
\end{lemma}

\begin{proof}
The ``only if" implication is obvious. Suppose $M\in \D_{\Tors}(A)$. Since $\RR\Gamma_\m$ is the right adjoint of a fully faithful embedding, the unit map $\eta \co M \to \RR\Gamma_\m(M)$ is an isomorphism in $\D(A)$. By the definition of an adjunction, the composition $\RR\Gamma_\m(M) \xra{\varepsilon} M \xra{\eta} \RR\Gamma_\m(M)$ is a quasi-isomorphism, and so $\varepsilon$ is as well. 
\end{proof}

Let $\n$ denote the homogeneous maximal ideal of $A^0$, and let $\RR\Gamma_\n$ denote the derived $\n$-torsion functor for $A^0$. If $M$ is an $H^0(A)$-module, then there is no difference between the derived $\n$-torsion of $M$ relative to $A^0$ or $H^0(A)$, and so we use the same notation for both. 

\begin{prop}[\cite{serreduality_dgalgebras} Proposition 3.7]
\label{prop:n}
If $M \in \D(A)$, then  $\RR\Gamma_\m(M) \cong \RR\Gamma_\n(M)$ in $\D(A^0)$.
\end{prop}

We recall from~\cite{serreduality_dgalgebras} a certain Ext-vanishing condition for dg-algebras:

\begin{dfn}
\label{def:chi}
We say $A$ satisfies \emph{condition~$\chi$} if, for all $M\in \Db(A)$ and~$j \in \Z$, we have $\uExt^j_A(\k, M)_i = 0$ for $i \gg 0$. 
\end{dfn}

We will show that Definition~\ref{def:chi} recovers Artin-Zhang's condition $\chi$ for ordinary algebras~\cite[Definition 3.7]{AZ} when $A = A^0$; see Remark~\ref{rem:chi}. The following result is an extension of \cite[Lemma~3.4 and Proposition~3.5(1)]{AZ}. 

\begin{lemma}
\label{lem:localcoh}
Let $M \in \Db(A)$ and $j \in \Z$.  
\begin{enumerate}
\item Let $T \in \Db(A)$ be a nonzero torsion object such that $H^m(T) = 0$ for $m > 0$, and set
$$
\lambda  \ce \inf\{i \text{ : } \Ext^m_A(\kk, M)_i \ne 0 \text{ for } m \le j\}, \quad \rho  \ce  \sup\{i \text{ : } \Ext^m_A(\kk, M)_i \ne 0 \text{ for } m \le j\},
$$
$$
\ell  \ce \inf\{i \text{ : } H(T)_i \ne 0\}, \quad r  \ce \sup\{i \text{ : } H(T)_i \ne 0 \} .
$$
For all $m \le j$, we have $\Ext^m_A(T, M)_i = 0$ for $i < \lambda - r$ and $i > \rho - \ell$. 
\item Assume $A$ satisfies condition $\chi$, let $N \in \Db(A)$ be such that $H^m(N) = 0$ for $m > 0$, and fix~$i \in \Z$. For $d_0 \gg 0$ and $d \ge d_0$, the canonical map
$$
\uExt^j_A(N/N_{\ge d_0}, M)_{\ge i} \to \uExt^j_A(N/N_{\ge d}, M)_{\ge i}
$$
is an isomorphism. In particular, there is an isomorphism
$
H^j_\m(M)_{\ge i} \cong \uExt^j_A(A/A_{\ge d_0}, M)_{\ge i}. 
$
\end{enumerate}
\end{lemma}

\begin{proof}
For (1), we argue by induction on $r-\ell$. Suppose $r-\ell= 0$, and say $H(T)_s \ne 0$. There is an isomorphism $T \cong T_{\ge s} / T_{\ge s+1}$ in $\D(A)$, and so $T$ is isomorphic in $\D(A)$ to a direct sum of objects of the form $\k(-s)[t]$ for $t \ge 0$. In this case, for all $m \le j$, we have $\Ext_A^m(T, M)_i = 0$ for $i < \lambda - s$ and $i > \rho - s$. 
If $r-\ell> 0$, let $s \ce \min\{i \text{ : } H(T)_i \ne 0\}$. The result follows by analyzing the long exact sequence obtained by applying $\Ext^*_A(- M)$ to the short exact sequence $0 \to T_{\ge s+1} \to T \to T / T_{\ge s+1} \to 0$. As for (2): for any integers $d \ge d_0$, we have an exact sequence
$$
\uExt^{j-1}_A(N_{\ge d_0}/N_{\ge d}, M) \to \uExt^j_A(N/N_{\ge d_0}, M) \to \uExt^j_A(N/N_{\ge d}, M) \to \uExt^j_A(N_{\ge d_0}/N_{\ge d}, M).
$$
Since condition $\chi$ holds, (1) implies $\underset{d_0 \to \infty}{\lim} \sup\{m \text{ : } \uExt^{t}_A(N_{\ge d_0}/N_{\ge d}, M)_m \ne 0 \text{ for } t \le j \} = - \infty$.
The statement immediately follows. 
\end{proof}

The following is an extension of \cite[Corollary 3.6(1)]{AZ}.

\begin{lemma}
\label{lem:chij}
Let $M \in \Db(A)$ and $T$ a torsion object in $\Db(A)$ such that $H^m(T) = 0$ for $m > 0$. If  $\Ext_A^j(\kk, M)_i = 0$ for $i \gg 0$ and all $j \le j_0$, 
then $\Ext_A^j(T, M)_i = 0$ for $i \gg 0$ and all $j \le j_0$. 
\end{lemma}

\begin{proof}
We may assume $T$ is not exact. Fix $j \le j_0$, and let $m \ce \min\{i \text{ : } H(T)_i \ne 0\}$. We argue by induction on
$
d \ce \max\{s - t \text{ : } H(T)_s \ne 0 \text{ and } H(T)_t \ne 0\}.
$
As argued in the proof of Lemma~\ref{lem:localcoh}, if $d = 0$, then $T$ is isomorphic in $\D(A)$ to a direct sum of objects of the form $\k(-s)[t]$ for $s \in \Z$ and $t \ge 0$, and we are done. 
Suppose $d > 0$. We have $\uExt_A^j(T_{\ge m+1}, M)_i = 0$ for $i \gg 0$ by induction, and $\uExt_A^j(T / T_{\ge m+1}, M)_i = 0$ for $i \gg 0$ by the $d = 0$ case. The triangle
$$
\RuHom_A(T/T_{\ge m+1}, M) \to \RuHom_A(T, M) \to \RuHom_A(T_{\ge m+1}, M) \to
$$
now implies the result. 
\end{proof}

The following series of equivalent characterizations of condition $\chi$ extend several results of Artin-Zhang \cite[Corollary 3.6(2) and (3), Proposition 3.8(1) and (3)]{AZ}.

\begin{prop}
\label{prop:chi}
The following are equivalent:
\begin{enumerate}
\item Condition $\chi$ holds for $A$.
\item If $M, N \in \Db(A)$, $\dim_\k H(M) < \infty$, and $j \in \Z$; then $\uExt^j_A(M, N)_i = 0$ for $i\gg 0$. 
\item For all $M \in \Db(A)$ and $i, j \in \Z$, there is some $d_0 \in \Z$ such that $\uExt^j_A(A/A_{\ge d}, M)_{\ge i}$ is a finitely generated $H^0(A)$-module for all $d \ge d_0$. 
\item For all $M \in \Db(A)$ and $j \in \Z$, $H^j_\m(M)_{i} = 0$ for $i \gg 0$.
\item For all $M \in \Db(A)$ and $i, j \in \Z$, $H^j_\m(M)_{\ge i}$ is a finitely generated $H^0(A)$-module. 
\end{enumerate}
\end{prop}

\begin{remark}
\label{rem:chi}
The equivalence of (1) and (3) in Proposition~\ref{prop:chi} implies that our definition of condition~$\chi$ recovers Artin-Zhang's condition $\chi$ for ordinary graded algebras \cite[Definition 3.7]{AZ}.
\end{remark}

\begin{proof}[Proof of Proposition~\ref{prop:chi}]
We have (1) implies~(2) by Lemma~\ref{lem:chij}, and (2) implies (3) by \Cref{prop:extfinite}(1). We now show (3) implies (1). Let $M \in \Db(A)$; we argue by induction on $j$ that $\uExt^j_A(\kk, M)_i = 0$ for $i \gg 0$. 
By \cite[Lemma 5.5]{serreduality_dgalgebras}, we have $\uExt^j_A(\kk, M) = 0$ for $j < t \ce \inf(M)$.  For all $j \in \Z$, (3) implies that there exists some $d_j \gg 0$ such that $\uExt_A^j(A/A_{\ge d}, M)_i = 0$ for $i \gg 0$. Moreover, for all $j \in \Z$, we have an exact sequence
\begin{equation}
\label{eqn:ind}
\uExt^{j-1}_A(A_{\ge 1}/A_{\ge d_j}, M) \to \uExt^{j}_A(\k, M) \to \uExt^{j}_A(A/A_{\ge d_j}, M).
\end{equation}
Applying \cite[Lemma 5.5]{serreduality_dgalgebras} once again, we have $\uExt^{j}_A(A_{\ge 1}/A_{\ge d}, M) = 0$ for $j < t$. Thus, the exact sequence \eqref{eqn:ind} gives an injection $\uExt^{t}_A(\k, M) \into \uExt^{t}_A(A/A_{\ge d_t}, M)$. We conclude that $\uExt^{t}_A(\k, M)_i = 0$ for $i \gg 0$. For $j > 0$, Lemma~\ref{lem:chij}, along with the sequence \eqref{eqn:ind} and induction, gives $\uExt^{j}_A(\k, M)_i = 0$ for $i \gg 0$.

Since (3) implies (1), we may invoke Lemma~\ref{lem:localcoh}(2) to conclude that (3) implies (5). It is also immediate that (5) implies (4). Finally, we show (4) implies~(1). Let $M \in \Db(A)$, and set $N^j \ce \underset{d \to \infty}{\colim} \text{ } \uExt^{j}_A(A_{\ge 1} / A_{\ge d}, M)$. For all $j \in \Z$, we have an exact sequence
\begin{equation}
\label{eqn:ext}
N^{j-1} \to \uExt^{j}_A(\k, M) \to H^j_\m(M).
\end{equation}

We prove by induction on $j$ that $\uExt^{j}_A(\k, M)_i = 0$ for $i \gg 0$. It follows from \cite[Lemma 5.5]{serreduality_dgalgebras} that $N^j = 0$ for $j \ll 0$. If $N^j = 0$ for all~$j$, then we are done; otherwise, let $j_0 \ce \min \{j \text{ : } N^j \ne 0\}$. The sequence~\eqref{eqn:ext} gives an injection $\uExt^{j_0}_A(\k, M) \into H^{j_0}_{\m}(M)$, and so $\uExt^{j_0}_A(\k, M)_i = 0$ for $i \gg 0$. If $j > j_0$, then Lemma~\ref{lem:chij} implies  $N^{j-1}_i = 0$ for $i \gg 0$, and so $\Ext^j_A(\k, M)_i = 0$ for~$i \gg 0$. 
\end{proof}

\begin{thm}
\label{thm:chi}
Condition $\chi$ holds for $A$ if and only if it holds for $H^0(A)$. 
\end{thm}

\begin{proof}
Suppose condition $\chi$ holds for $A$. Let $M$ be a finitely generated $H^0(A)$-module. As discussed above, the module $M$ determines an object in $\Db(A)$ via restriction of scalars along the dg-algebra morphism $A \to H^0(A)$. By the equivalence of (1) and (4) in Proposition~\ref{prop:chi}, $H^j_\m(M)_i = 0$ for $i \gg 0$ and all $j \in \Z$. Recall that $\n$ denotes the homogeneous maximal ideal of $H^0(A)$; by Proposition~\ref{prop:n}, we have $H^*_\n(M) \cong H^*_\m(M)$. Thus, $H^j_\n(M)_i = 0$ for $i \gg 0$ and all $j \in \Z$. Since finitely generated $H^0(A)$-modules generate $\Db(H^0(A))$, Proposition~\ref{prop:chi} implies that condition $\chi$ holds for~$H^0(A)$. Conversely, say condition $\chi$ holds for $H^0(A)$. By Proposition~\ref{prop:gen}, we need only show $H^j_\m(M)_i = 0$ for $i \gg 0$ when $M$ is a finitely generated $H^0(A)$-module, which follows once again from \Cref{prop:n}.
\end{proof}

Given a left dg-$A$-module $M$ (i.e. a right dg-$A^{\op}$-module), we let $\RR\Gamma_{\m^{\op}}(M)$ denote its derived $\m$-torsion. If $M$ is a dg-$A$-$A$-bimodule, then both $\RR\Gamma_\m(M)$ and $\RR\Gamma_{\m^{\op}}(M)$ are also dg-$A$-$A$-bimodules~\cite[Remark 3.6]{serreduality_dgalgebras}.

\begin{thm}
\label{thm:op}
Let $M$ be a dg-$A$-$A$-bimodule such that $M\in \Db(A)$ and $M\in \Db(A^{\op})$. If $A$ and~$A^{\op}$ satisfy condition $\chi$, and $H^0(A)$ and $H^0(A)^{\op}$ have finite local cohomological dimension, then there is an isomorphism $\RR\Gamma_\m(M) \cong \RR\Gamma_{\m^{\op}}(M)$ in $\D(A^{\op} \otimes_\kk A)$. 
\end{thm}

\begin{proof}
Set $A^e \ce A^{\op} \otimes_\kk A$, and let $\RR\Gamma_{\m^e}(-):\D(A^e)\to \D(A^e)$ denote the associated derived torsion functor. We caution the reader that  $H^0(A^e)$ need not be Noetherian (this can occur even when $A = A^0$), and so we have left the context of Setup~\ref{setup} only in this proof. However, the definition of derived torsion still makes sense for $A^e$. 

For $N\in \D(A)$, the functor $- \otimes_\kk N \co \D(A^{\op}) \to \D(A^e)$ admits a right adjoint  $\RuHom_A(N,-)$. We have the following natural isomorphisms in $\D(A^e)$:
\begin{align*}
\RR\Gamma_{\m^e}(M) &\cong \underset{d \to \infty}{\colim}\text{ } \RuHom_{A^e}(A^e / A^e_{\ge d}, M)\\
& \cong \underset{d \to \infty}{\colim}\text{ } \RuHom_{A^e}(A^{\op} / A^{\op}_{\ge d}\otimes_{\kk} A/A_{\ge d}, M) \\
& \cong \underset{d \to \infty}{\colim}\text{ } \RuHom_{A^{\op}}(A^{\op} / A^{\op}_{\ge d}, \RuHom_A(A/A_{\ge d}, M)) \\
& \cong \underset{d \to \infty}{\colim} \text{ }\underset{e \to \infty}{\colim}\text{ } \RuHom_{A^{\op}}(A^{\op} / A^{\op}_{\ge d}, \RuHom_A(A/A_{\ge e}, M)) \\
& \cong \underset{d \to \infty}{\colim}\text{ } \RuHom_{A^{\op}}(A^{\op} / A^{\op}_{\ge d}, \underset{e \to \infty}{\colim} \text{ }\RuHom_A(A/A_{\ge e}, M))\\
& \cong \RR\Gamma_{\m^{\op}}(\RR\Gamma_{\m}(M)),
\end{align*}
where the fifth isomorphism follows from \Cref{lem:colimits}. By symmetry, we also have an isomorphism $\RR\Gamma_{\m^e}(M)\cong \RR\Gamma_{\m}( \RR\Gamma_{\m^{\op}}(M))$ in $\D(A^e)$, and so we have $\RR\Gamma_{\m^{\op}}(\RR\Gamma_{\m}(M)) \cong \RR\Gamma_{\m}( \RR\Gamma_{\m^{\op}}(M))$ in $\D(A^e)$. It follows from the equivalence of (1) and (4) in \Cref{prop:chi} that $\RR\Gamma_{\m}(M)\in \D_{\Tors}(A^{\op})$ and $\RR\Gamma_{\m^{\op}}(M)\in \D_{\Tors}(A)$. Since the counit maps $\RR\Gamma_{\m}(\RR\Gamma_{\m^{\op}}(M)) \to \RR\Gamma_{\m^{\op}}(M)$ and $\RR\Gamma_{\m^{\op}}(\RR\Gamma_{\m}(M)) \to \RR\Gamma_{\m}(M)$ are $A^e$-linear~\cite[Remark 3.6]{serreduality_dgalgebras},  \Cref{lem:counit} finishes the proof.
\end{proof}

\subsection{Balanced dualizing dg-modules}
\label{sec:balanced}

Given a dg-$A$-$A$-bimodule $R$, a dg-$A$-module $M$ is called \emph{$R$-reflexive} if the natural map
$
M \to \RuHom_{A^{\op}}(\RuHom_A(M, R), R)
$
is a quasi-isomorphism. 

\begin{defn}[\cite{serreduality_dgalgebras} Definitions 3.8 and 3.13]
 A dg-$A$-$A$-bimodule $R$ is called \emph{dualizing} if, for any $M \in \Db(A^{\op})$ and $N \in \Db(A)$, the following conditions are satisfied:
\begin{enumerate}

\item We have $\RuHom_{A^{\op}}(M, R) \in \Db(A)$, and $\RuHom_A(N, R) \in \Db(A^{\op})$.
\item The dg-$A$-modules $N$ and $N \otimes_A R$, and dg-$A^{\op}$-modules  $M$ and $M \otimes_{A^{\op}} R$, are $R$-reflexive.
\end{enumerate}
A \emph{balanced dualizing dg-$A$-module} is a dualizing dg-$A$-module $R$ such that there are isomorphisms $\RR\Gamma_{\m}(R) \cong A^* \cong \RR\Gamma_{\m^{\op}}(R)$ in $\D(A^{\op} \otimes_\k A)$.
\end{defn}

As discussed in the introduction, the existence of a balanced dualizing dg-module for $A$ yields a notion of Serre duality for the noncommutative space associated to $A$: see Section~\ref{sec:serre} for details. By~\cite[Proposition 3.15]{serreduality_dgalgebras}, if $\dim_\k H(A) <~\infty$, then $R = A^*$ is a balanced dualizing dg-module for $A$. When $A$ is graded commutative (i.e. $xy = (-1)^{|x||y|}yx$ for all homogeneous $x, y \in A$) and Gorenstein~(\Cref{def:gorenstein}), it is proven in~\cite[Theorem 3.18]{serreduality_dgalgebras} that $R = A(-a)[n]$ is a balanced dualizing dg-module for $A$, where $a$ and $n$ are as in Definition~\ref{def:gorenstein}. 



We now state our main theorem, which is a more precise formulation of Theorem~\ref{thm:intro}. It characterizes what it means for a dg-algebra as in Setup~\ref{setup} to have a balanced dualizing dg-module, generalizing a theorem of Van den Bergh~\cite[Theorem 6.3]{VdB_dualizing}. It also provides a much larger family of examples of dg-algebras with balanced dualizing dg-modules: see Corollary~\ref{cor:A0}. 

\begin{thm}
\label{thm:vdb}
 If $A$ has a balanced dualizing dg-module $R$, then there is an isomorphism $R \cong \RR\Gamma_\m(A)^*$ in $\D(A^{\op} \otimes_\k A)$, and we have the following:
 \begin{enumerate}
 \item $H^0(A)$ and $H^0(A)^{\op}$ satisfy condition $\chi$.
 \item $\lcd(H^0(A))<\infty$ and $\lcd(H^0(A)^{\op})<\infty$ (see \Cref{def:lcd_classical}).
 \end{enumerate}
Conversely, if (1) and~(2) hold, then $\RR\Gamma_\m(A)^*$ is a balanced dualizing dg-module for $A$. 
\end{thm}
 
Theorem~\ref{thm:vdb} implies that balanced dualizing dg-modules are unique up to isomorphism in $\D(A^{\op} \otimes_\k A)$. Before proving Theorem~\ref{thm:vdb}, we establish some intermediate technical results. The first is a generalization of \cite[Proposition 2.1]{jorgensen97}:

\begin{prop}
\label{prop:tensor}
Suppose $H^0(A)$ has finite local cohomological dimension, and let $M \in \D(A)$. 
\begin{enumerate}
\item There is an isomorphism
$
M \otimes_A^\LL \RR\Gamma_\m(A) \xra{\cong} \RR\Gamma_\m(M)
$ in $\D(A)$.
\item If $M \in \D(A^{\op} \otimes_\k A)$, then the isomorphism in (1) may be chosen to be $A^{\op} \otimes_\k A$-linear. 
\end{enumerate}
\end{prop}

\begin{proof}
Let $F$ be a semifree resolution of $M$. We may assume $F = M$. For all $d \ge 0$, choose a semifree resolution $F_d$ of $A / A_{\ge d}$ as an $A^{\op} \otimes_\k A$-module. Each $F_d$ is also a semifree resolution over~$A$. For each $d \ge 0$, there is a morphism $\varphi_d \co M \otimes_A  \uHom_A(F_d, A) \to \uHom_A(F_d, M)$ of dg-$A$-modules that sends $y \otimes \varphi  \in M \otimes_A \uHom_A(F_d, A) $ to the map $F_d \to M$ given by $x \mapsto y\varphi_d(x)$. For $d \le d'$, we have maps $F_{d'} \to F_d$ induced by the surjections $A/A_{\ge d'} \to A/A_{\ge d}$, yielding a commutative square
$$
\xymatrix{
 M\otimes_A \uHom_A(F_d, A)  \ar[r]^-{\varphi_d} \ar[d] & \uHom_A(F_d, M) \ar[d]\\
M\otimes_A \uHom_A(F_{d'}, A) \ar[r]^-{\varphi_{d'}} & \uHom_A(F_{d'}, M). 
}
$$
Passing to colimits, we obtain a map $\a \co M \otimes_A \RR\Gamma_\m(A)   \to \RR\Gamma_\m(M)$. The map $\a$ is natural in $M$, and it is a quasi-isomorphism when the rank of $M$ as a free $A$-module is 1. It follows by induction on the rank of $M$ that $\a$ is a quasi-isomorphism when $M$ is a perfect object \cite[Definition 2.14]{BROWN2025110035}. Since $M$ is a filtered colimit of perfect objects, applying \Cref{lem:colimits} finishes the proof of (1). To prove (2), choose $F$ to be a semifree resolution of $M$ over $A^{\op}\otimes_\k A$ and proceed as in~(1). 
\end{proof}


As a consequence of Proposition~\ref{prop:tensor}, we obtain the following extension of \cite[Theorem 5.1(2)]{VdB_dualizing} (see also \cite[Theorem 2.3]{jorgensen97}), which is a variant of the local duality theorem for differential bigraded algebras given in \cite[Theorem 3.11]{serreduality_dgalgebras}. 

\begin{cor}
\label{cor:LD}
Suppose that $H^0(A)$ has finite local cohomological dimension. For $M \in \D(A)$, we have $\RR\Gamma_{\m}(M)^* \xra{\cong} \RR\uHom_A(M,\dGamma_{\m}(A)^*)$ in $\D(A^{\op})$. If $M$ is a dg-$A$-$A$-bimodule, then this isomorphism may be chosen to be $A^{\op}\otimes_\k A$-linear. 
\end{cor}

\begin{proof}
We have isomorphisms $\RR\Gamma_{\m}(M)^* \xra{\cong} (M \otimes_A^\LL \RR\Gamma_\m(A) )^* \xra{\cong} \RR\uHom_A(M,\dGamma_{\m}(A)^*)$ in $\D(A^{\op})$, where the first follows from Proposition~\ref{prop:tensor} and the second from adjunction. This proves the first statement, and the second follows from Proposition~\ref{prop:tensor}(2). 
\end{proof}

\begin{prop}
\label{prop:fg}
Assume $H^0(A)$ satisfies condition $\chi$. If $M \in \Db(A)$, then $H^j_{\m}(M)^*$ is finitely generated over $H^0(A)^{\op}$ for all $j \in \Z$. In particular, $\dim_\k H^j_{\m}(M)_i < \infty$ for all $i, j \in \Z$. 
\end{prop}

\begin{proof}
By Proposition~\ref{prop:gen}, we may assume $M$ is a finitely generated $H^0(A)$-module. Fix $j \in \Z$, and recall that $\n$ denotes the homogeneous maximal ideal of $A^0$. We have $H^j_{\m}(M) \cong H^j_{\n}(M)$ by Proposition~\ref{prop:n}. Since $H^0(A)$ satisfies condition $\chi$, the result follows from \cite[Proposition 7.2 and Proposition 7.9]{AZ}.
\end{proof}


\begin{prop} 
\label{AsimeqRHomRR} 
Let $R \ce \RR\Gamma_\m(A)^*$. If both $H^0(A)$ and $H^0(A)^{\op}$ satisfy condition~$\chi$, and both $H^0(A)$ and $H^0(A)^{\op}$ have finite local cohomological dimension, then $A \cong\RR\uHom_A(R, R)$ in $\D(A)$. 
\end{prop}

\begin{proof} 
Applying Lemma~\ref{lem:matlis} and Proposition~\ref{prop:fg}, we have isomorphisms
\begin{equation}
\label{eqn:RR}
\RuHom_A(R, R) \cong \RuHom_{A^{\op}}(R^*, R^*) \cong \RuHom_{A^{\op}}(\RR\Gamma_\m(A), \RR\Gamma_\m(A))
\end{equation}
in $\D(A^{\op} \otimes_\k A)$. By Theorems~\ref{thm:chi} and~\ref{thm:op}, there is an isomorphism $\RR\Gamma_\m(A) \cong \RR\Gamma_{\m^{\op}}(A)$ in $\D(A \otimes_\kk A^{\op})$. In particular, $\RR\Gamma_\m(A) \in \D_{\Tors}(A^{\op})$, and so the adjunction between the inclusion $\D_{\Tors}(A^{\op}) \into \D(A^{\op})$ and the functor $\RR\Gamma_{\m^{\op}}$ implies that there is an isomorphism
\begin{equation}
\label{eqn:isofun}
\RuHom_{A^{\op}}(\RR\Gamma_\m(A), A ) \cong \RuHom_{A^{\op}}(\RR\Gamma_\m(A), \RR\Gamma_{\m^{\op}}(A))
\end{equation}
in $\D(A^{\op} \otimes_\k A)$. 
We therefore have the following isomorphisms in $\D(A \otimes_\kk A^{\op})$:
\begin{align*}
\RuHom_A(R, R) & \cong \RuHom_{A^{\op}}(\RR\Gamma_\m(A), \RR\Gamma_\m(A)) \\
& \cong \RuHom_{A^{\op}}(\RR\Gamma_\m(A), \RR\Gamma_{\m^{\op}}(A)) \\
& \cong \RuHom_{A^{\op}}(\RR\Gamma_\m(A), A) \\
& \cong \RuHom_{A}(A^*,R) \\
& \cong \dGamma_{\m}(A^*)^* \\
& \cong (A^*)^* \\
& \cong A.
\end{align*}
The first follows from~\eqref{eqn:RR}, the second from Theorems~\ref{thm:chi} and~\ref{thm:op}, the third from~\eqref{eqn:isofun}, the fourth from Lemma~\ref{lem:matlis} and Proposition~\ref{prop:fg}, and the fifth from Corollary~\ref{cor:LD}. The sixth isomorphism holds by Lemma~\ref{lem:counit}, since $A^* \in \D_{\Tors}(A)$.
\end{proof}

\begin{proof}[Proof of Theorem~\ref{thm:vdb}]
Let $R$ be a balanced dualizing dg-module for $A$. By \cite[Corollary 3.14]{serreduality_dgalgebras} and adjunction, there is an isomorphism $ \alpha \co \RR\Gamma_\m(A) \cong R^*$ in $\D(A)$. The map $\alpha$ is moreover an isomorphism in $\D(A^{\op} \otimes_\k A)$; this is because, when the object $M$ in the statement of the local duality theorem \cite[Theorem 3.11]{serreduality_dgalgebras} is a dg-$A$-$A$-bimodule, the local duality isomorphism in that statement is $A^{\op} \otimes_\k A$-linear. Thus, $R \cong \RR\Gamma_\m(A)^*$ in $\D(A^{\op} \otimes_\k A)$. It follows from \cite[Theorem 5.1(2)]{serreduality_dgalgebras}\footnote{It is assumed in \cite[Theorem 5.1]{serreduality_dgalgebras} that $A$ is Gorenstein, but this is not needed for part (2) of that theorem.}, the exact triangle in \cite[(4.3)]{serreduality_dgalgebras}, and \Cref{prop:n} that we have $\lcd(H^0(A)) \le -\inf(R)$ and $\lcd(H^0(A)^{\op}) \le -\inf(R)$.  For any $M \in \Db(A)$, \cite[Corollary 3.14]{serreduality_dgalgebras} implies that there is an isomorphism $\RR\Gamma_\m(M) \cong \RuHom_A(M, R)^*$. Since $R$ is a dualizing dg-module, $\RuHom_A(M, R) \in \Db(A)$. We conclude that $H^j_\m(M)_i = 0$ for $i \gg 0$, which implies condition $\chi$ for $A$ by the equivalence of (1) and (4) in Proposition~\ref{prop:chi}. A similar argument shows that condition $\chi$ holds for $A^{\op}$. \Cref{thm:chi} thus implies $H^0(A)$ and $H^0(A)^{\op}$ satisfy condition~$\chi$.

Conversely, let $R \ce \RR\Gamma_\m(A)^*$, and suppose conditions (1) and (2) in the statement hold. Let $M \in \Db(A^{\op})$ and $N \in \Db(A)$. By Corollary~\ref{cor:LD}, we have $\RuHom_A(N, R) \cong \RR\Gamma_\m(N)^*$ in $\D(A^{\op})$. Since $H^0(A)$ has finite local cohomological dimension and satisfies $\chi$, Propositions~\ref{prop:gen} and \ref{prop:fg} imply that $\RuHom_A(N, R) \in \Db(A^{\op})$. The same argument shows $\RuHom_{A^{\op}}(M, R) \in \Db(A)$. By Proposition~\ref{prop:fg}, we have both $R \in \Db(A)$ and $R \in \Db(A^{\op})$; thus, to prove that $R$ is a dualizing dg-module, we need only show $M$ and $N$ are $R$-reflexive. We show $M$ is $R$-reflexive; the proof for~$N$ is the same. We have isomorphisms
\begin{align*}
\RuHom_A(\RuHom_{A^{\op}}(M, R), R) & \cong \RuHom_A(\RR\Gamma_{\m^{\op}}(M)^*, R) \\
&\cong \RuHom_{A^{\op}}(\RR\Gamma_\m(A), \RR\Gamma_{\m^{\op}}(M)) \\
&\cong \RuHom_{A^{\op}}(\RR\Gamma_\m(A), M) \\
&\cong \RuHom_{A}(M^*, R) \\
&\cong \RR\Gamma_\m(M^*)^*\\
&\cong (M^*)^*\\
&\cong M
\end{align*}
in $\D(A^{\op})$, where the first follows from Corollary~\ref{cor:LD}, the second from Lemma~\ref{lem:matlis} and Proposition~\ref{prop:fg}, the third from adjunction and Theorem~\ref{thm:op}, the fourth from Lemma~\ref{lem:matlis} and Proposition~\ref{prop:fg} once again, the fifth from Corollary~\ref{cor:LD}, and the sixth by Lemma~\ref{lem:counit}, since $M^* \in \D_{\Tors}(A)$. Moreover, this isomorphism $M \cong \RuHom_A(\RuHom_{A^{\op}}(M, R), R)$ is given by the canonical evaluation map; to verify this, one may replace $M$ with a semifree resolution, check it in the case where $M = A$, and then pass to colimits. By Corollary~\ref{cor:LD} and Proposition~\ref{AsimeqRHomRR}, we have 
$
\RR\Gamma_\m(R)^* \cong \RuHom_A(R, R) \cong A
$
in $\D(A \otimes_\kk A^{\op})$, and so $R$ is a balanced dualizing dg-module. 
\end{proof}


The following consequence of Theorem~\ref{thm:vdb} implies Corollary~\ref{cor:intro}:

\begin{cor}
\label{cor:A0}
The following hold:
\begin{enumerate}
\item  The dg-algebra $A$ admits a balanced dualizing dg-module $R$ if and only if $H^0(A)$ admits a balanced dualizing complex $R'$, and $R' \cong \RuHom_A(H^0(A), R)$ in $\D(A^{\op} \otimes_k A)$. 
\item If $A^0$ is Noetherian and admits a balanced dualizing complex $R$, then $A$ admits a balanced dualizing dg-module $R'$, and $R' \cong \RuHom_{A^0}(A, R)$ in $\D((A^0)^{\op} \otimes_\k A^0)$.
\end{enumerate}
\end{cor}

\begin{proof}
It is immediate from \Cref{thm:vdb} and~\cite[Theorem 6.3]{VdB_dualizing} that $A$ admits a balanced dualizing dg-module $R$ if and only if $H^0(A)$ admits a balanced dualizing complex $R'$, and the isomorphism $R'\cong \RuHom_A(H^0(A), R)$  follows from Corollary~\ref{cor:LD}. Suppose $A^0$ is Noetherian and admits a balanced dualizing complex $R$. By Theorem~\ref{thm:vdb}, condition~$\chi$ holds for $A^0$ and $(A^0)^{\op}$. By the equivalence of (1) and (4) in Proposition~\ref{prop:chi}, condition $\chi$ also holds for $H^0(A)$ and $H^0(A)^{\op}$, and so Theorem~\ref{thm:chi} implies that condition $\chi$ holds for $A$ and $A^{\op}$. Finally, every object in $\Db(H^0(A))$ is also an object in $\Db(A^0)$. Thus, by Proposition~\ref{prop:n} and Theorem~\ref{thm:vdb}, we have $\lcd(H^0(A)) < \infty$. One argues similarly that $\lcd(H^0(A)^{\op}) < \infty$. Thus, Theorem~\ref{thm:vdb} yields a balanced dualizing dg-module $R'$ of $A$, and Corollary~\ref{cor:LD} implies that there is an isomorphism $R' \cong \RuHom_{A^0}(A, R)$ in $\D((A^0)^{\op} \otimes_\k A^0)$. 
\end{proof}

\begin{remark}
 If the equivalent conditions of Theorem~\ref{thm:vdb} hold, and $R'$ is the balanced dualizing complex of $H^0(A)$, then we have $\lcd(H^0(A)) = \lcd(H^0(A)^{\op}) = -\inf(R')$~ \cite[Theorem 4.2(3)]{Yekutieli_Zhang}. 
\end{remark}

We record the following special case of Corollary~\ref{cor:A0}:
\begin{cor}
\label{cor:comm}
If $H^0(A)$ is Noetherian and commutative, then $A$ admits a balanced dualizing dg-module. 
\end{cor}

\begin{proof}
Condition $\chi$ holds for any commutative Noetherian algebra, and $\lcd(H(A)^0) < \infty$ is immediate as well. The result therefore follows from Theorem~\ref{thm:vdb}. 
\end{proof}

\begin{example}
\label{ex:explicit}
Corollary~\ref{cor:A0} provides many examples of dg-algebras with balanced dualizing dg-modules. For instance, by results of Yekutieli~\cite{yek92}, a Noetherian $\Z$-graded connected $\k$-algebra admits a balanced dualizing complex when it is Gorenstein in the sense of Definition~\ref{def:gorenstein}, a finitely generated $\k$-algebra that is finite over its center (e.g. an Azumaya algebra), or a skew homogeneous coordinate ring in the sense of \cite[Definition 6.6]{yek92}. Thus, if $A^0$ or $H^0(A)$ is of any of these forms, then $A$ admits a balanced dualizing dg-module. 

As an explicit example, let $S$ be a Noetherian $\Z$-graded connected $\k$-algebra, and suppose $f_1, \dots, f_c \in S$ are homogeneous elements that are central, i.e. such that $f_i s = sf_i$ for all $s \in S$. Let~$E$ denote the exterior algebra $\bigwedge_\k(e_1, \dots, e_c)$, and equip the $\k$-algebra $K \ce E \otimes_\k S$ with the internal grading $\deg(e_i \otimes s) = \deg(f_i) + \deg(s)$, cohomological grading  $|e_i \otimes s| = -1$, and $S$-linear differential  $\del_K(e_{i_1} \cdots e_{i_t}) = \sum_{j = 1}^t (-1)^{j-1}e_{i_1} \cdots \widehat{e_{i_j}} \cdots e_{i_t} \otimes f_j$ . The complex $K$ is called the \emph{Koszul complex on $f_1, \dots, f_c$}, and it is a dg-algebra satisfying the conditions in Setup~\ref{setup}. If $S$ or $H^0(K)$ admits a balanced dualizing complex, then Corollary~\ref{cor:A0} implies that $K$ admits a balanced dualizing dg-module. 
\end{example}

\begin{example}
\label{ex:gor}
Suppose $A$ is Gorenstein (Definition~\ref{def:gorenstein}) and that $A$ has finite injective dimension as a right dg-$A$-module, in the sense of \cite[Definition 2.1]{finitistic_dimensions}. Denote the injective dimension of $A$ by $\on{inj} \dim(A)$. Assume also that $H^0(A)$ is Noetherian and that it admits a balanced dualizing complex. Choose $a, n \in \Z$ such that there is an isomorphism $\RuHom_A(\k, A) \cong \k(a)[-n]$ in $\D(A)$. \Cref{thm:vdb} implies that $A$ admits a balanced dualizing dg-module $R \cong \RR\Gamma_\m(A)^*$. We now show that there is an isomorphism $R \cong A(-a)[n]$ in~$\D(A)$. It follows from \cite[Theorem 3.24]{minamoto}\footnote{In the notation of \cite[Theorem 3.24]{minamoto}, $\mathrm{id}(A)=\injdim(A)-\inf(A)$.} that $A(-a)[n]$ possesses a minimal injective (ifij) resolution (\cite[Definition 3.19]{minamoto}) of finite length. Following the procedure outlined in the second paragraph of the proof of \cite[Theorem 3.18]{serreduality_dgalgebras}, we can construct a $K$-injective resolution $J$ of $A(-a)[n]$ built out of the aforementioned minimal injective resolution. Proceeding as in paragraphs three and four of the proof of \cite[Theorem 3.18]{serreduality_dgalgebras}, one sees that that  $R \cong A(-a)[n]$ in $\D(A)$. We caution the reader that it need not be the case that $R \cong A(-a)[n]$ in $\D(A^{\op} \otimes_\k A)$, even when $A = A^0$; see e.g. \cite[Theorem 9.2]{VdB_dualizing}.
\end{example}

\begin{remark}
  The assumption that $\on{inj} \dim(A) < \infty$ in Example~\ref{ex:gor} is superfluous when $A$ is graded commutative: see the first paragraph of the proof of \cite[Theorem 3.18]{serreduality_dgalgebras}.
\end{remark}



\section{Serre duality}
\label{sec:serre}

Theorem~\ref{thm:vdb} implies that Serre duality holds for the noncommutative spaces associated to a host of dg-algebras. Before we state the result (Corollary~\ref{cor:serre}), we recall some background and notation concerning noncommutative geometry over a dg-algebra (see e.g. \cite[Section 4.1]{serreduality_dgalgebras}). Let $\D_{\Qgr}(A)$ denote the quotient $\D(A) / \D_{\Tors(A)}$, i.e. the derived category of the noncommutative space associated to $A$. Given $M \in \D(A)$, let $\widetilde{M}$ denote the image of $M$ under the canonical functor $\D(A) \to \D_{\Qgr}(A)$. We write 
$\Ext^j(\widetilde{M}, \widetilde{N}) \ce \Hom_{\D_{\Qgr}(A)}(\widetilde{M}, \widetilde{N}[j])$. By \cite[Proposition 4.5]{serreduality_dgalgebras}, for $M \in \Db(A)$ and $N \in \D(A)$, we have:
$$
\Ext^j(\widetilde{M}, \widetilde{N})  \cong \underset{d \to \infty}{\on{colim}} \text{ } \RHom_A(M_{\ge d}, N[j]).
$$
Letting $\OO \ce \widetilde{A}$, we write $\RR^j\Gamma(\widetilde{M}) \ce \Ext^j(\OO, \widetilde{M})$. For $M \in \D(A)$ and $i \in \Z$, we set $\widetilde{M}(i) \ce \widetilde{M(i)}$.

\begin{cor}[Serre duality]
\label{cor:serre}
Suppose $H^0(A)$ and $H^0(A)^{\op}$ satisfy condition $\chi$, and assume $\lcd(H^0(A)) < \infty$ and $\lcd(H^0(A)^{\op}) <~\infty$. For all $M \in \Db(A)$ and all $j \in \Z$, there is a natural isomorphism $\RR^j\Gamma(\widetilde{M}) \cong \Ext^{-j-1}(\widetilde{M}, \widetilde{R})^*$, where $R$ is the balanced dualizing dg-module of $A$. 
\end{cor}

\begin{proof}
The existence of the balanced dualizing dg-module $R$ is guaranteed by Theorem~\ref{thm:vdb}, and the isomorphism follows from \cite[Theorem 1.3]{serreduality_dgalgebras}.
\end{proof}

We now turn to proving Corollary~\ref{cor:serrefunctor}. Let $\D_{\qgr}(A)$ denote the essential image of $\Db(A)$ under the canonical functor $\D(A) \to \D_{\Qgr}(A)$, and let $\Perf_{\qgr}(A)$ be the subcategory of $\D_{\qgr}(A)$ generated by compact objects. We call $\Perf_{\qgr}(A)$ the category of \emph{perfect complexes} in $\D_{\qgr}(A)$. 

\begin{defn}
Let $\mathcal{T}$ be a triangulated category. 
\begin{enumerate}
\item The category $\mathcal{T}$ is called \emph{proper} if $\dim_\k \Hom_{\mathcal{T}}(T, T') < \infty$ for all $T, T' \in \mathcal{T}$. 
\item Assume $\mathcal{T}$ is proper. A functor $S \co \mathcal{T} \to \mathcal{T}$ is called a \emph{Serre functor} if there is a natural isomorphism $\Hom_{\mathcal{T}}(T, T') \cong \Hom_{\mathcal{T}}(T', S(T))^*$ for all objects $T, T' \in \mathcal{T}$.
\end{enumerate}
\end{defn}

\begin{prop}
\label{prop:proper}
The following hold:
\begin{enumerate}
\item If $A$ satisfies condition $\chi$, then $\D_{\qgr}(A)$ is proper.
\item If $\lcd(H^0(A)) < \infty$, then the category $\Perf_{\qgr}(A)$ is classically generated by $\{\OO(i) \text{ : } i \in \Z\}$.
\end{enumerate}
\end{prop}

\begin{proof}
We have an exact sequence
$$
\RHom_A(M, N) \to \Hom_{\D_{\qgr}(A)}(\widetilde{M}, \widetilde{N}) \to \underset{d \to \infty}{\on{colimit}} \text{ } \RHom_A(M/M_{\ge d}, N[1]).
$$
The term on the left is finite dimensional by Proposition~\ref{prop:extfinite}(1), and the term on the right is finite dimensional by Proposition~\ref{prop:extfinite}(1) and Lemma~\ref{lem:localcoh}(2). This proves (1). To prove (2), we argue as in \cite[Lemma 4.2.2]{BV03}. It follows from \Cref{lem:colimits} and the exact triangle in \cite[(4.3)]{serreduality_dgalgebras} that each $\OO(i)$ is a compact object. 
Suppose $\widetilde{M}$ is in the right orthogonal of $\OO(i)$ for all $i$; that is, $\Hom_{\D_{\qgr}(A)}(\OO(i), \widetilde{M}) = 0$ for all $i$. The canonical functor $\D(A) \to \D_{\Qgr}(A)$ admits a fully faithful right adjoint $\RR\Gamma_*$~\cite[Proposition 3.3]{BROWN2025110035}. This adjunction implies that $\RR\Gamma_*(\widetilde{M}) = 0$; since $\RR\Gamma_*$ is fully faithful, we conclude that $\widetilde{M} = 0$. Part (2) now follows from \cite[Theorem 2.1.2]{BV03}, which Bondal-Van den Bergh attribute to Ravenel and Neeman~\cite{neeman}. 
\end{proof}

\begin{proof}[Proof of Corollary~\ref{cor:serrefunctor}]
By \Cref{thm:chi} and Proposition~\ref{prop:proper}(1), the category $\D_{\qgr}(A)$ is proper. By Example~\ref{ex:gor}, $A$ admits a balanced dualizing dg-module $R$, and we have $R \cong A(-a)[n]$ in~$\D(A)$. Fix $N \in \Db(A)$, and let $\on{Vect}(\k)$ denote the category of $\k$-vector spaces. We define functors
$F \co \Perf_{\qgr}(A) \to \on{Vect}(\k)$ and $G \co \Perf_{\qgr}(A) \to \on{Vect}(\k)$ given by $F(\widetilde{M}) = \Hom_{\D_{\Qgr}(A)}(\widetilde{M}, \widetilde{N})$ and $G(\widetilde{M}) = \Hom_{\D_{\Qgr}(A)}(\widetilde{N}, \widetilde{M}(-a)[n])$.
By Corollary~\ref{cor:serre}, there is a natural isomorphism $F(\OO(i)) \cong G(\OO(i))$ for all $i$. Thus, by Proposition~\ref{prop:proper}(2), $F$ and $G$ are isomorphic functors. 
\end{proof}

\bibliographystyle{amsalpha}
\bibliography{references}

\providecommand{\bysame}{\leavevmode\hbox to3em{\hrulefill}\thinspace}
\providecommand{\MR}{\relax\ifhmode\unskip\space\fi MR }
\providecommand{\MRhref}[2]{%
  \href{http://www.ams.org/mathscinet-getitem?mr=#1}{#2}
}
\providecommand{\href}[2]{#2}
\begin{thebibliography}{BSSW25}

\bibitem[AZ94]{AZ}
M.~Artin and J.~J. Zhang, \emph{Noncommutative projective schemes}, Adv. Math. \textbf{109} (1994), no.~2, 228--287. \MR{1304753}

\bibitem[BS24]{serreduality_dgalgebras}
Michael~K. Brown and Prashanth Sridhar, \emph{Serre duality for dg-algebras}, 2024.

\bibitem[BS25]{BROWN2025110035}
\bysame, \emph{Orlov's theorem for dg-algebras}, Advances in Mathematics \textbf{460} (2025), 110035.

\bibitem[BSSW25]{finitistic_dimensions}
Isaac Bird, Liran Shaul, Prashanth Sridhar, and Jordan Williamson, \emph{Finitistic dimensions over commutative {DG}-rings}, Math. Z. \textbf{309} (2025), no.~1, Paper No. 3, 29. \MR{4827125}

\bibitem[BVdB03]{BV03}
A.~Bondal and M.~Van~den Bergh, \emph{Generators and representability of functors in commutative and noncommutative geometry}, Mosc. Math. J. \textbf{3} (2003), no.~1, 1--36, 258. \MR{1996800}

\bibitem[J{\o}r97]{jorgensen97}
Peter J{\o}rgensen, \emph{Local cohomology for non-commutative graded algebras}, Comm. Algebra \textbf{25} (1997), no.~2, 575--591. \MR{1428799}

\bibitem[Min21]{minamoto}
Hiroyuki Minamoto, \emph{Resolutions and homological dimensions of {DG}-modules}, Israel J. Math. \textbf{245} (2021), no.~1, 409--454. \MR{4357467}

\bibitem[Nee92]{neeman}
Amnon Neeman, \emph{The connection between the {$K$}-theory localization theorem of {T}homason, {T}robaugh and {Y}ao and the smashing subcategories of {B}ousfield and {R}avenel}, Ann. Sci. \'{E}cole Norm. Sup. (4) \textbf{25} (1992), no.~5, 547--566. \MR{1191736}

\bibitem[{Sta}25]{stacks-project}
The {Stacks project authors}, \emph{The {S}tacks project}, \url{https://stacks.math.columbia.edu}, 2025.

\bibitem[VdB97]{VdB_dualizing}
Michel Van~den Bergh, \emph{Existence theorems for dualizing complexes over non-commutative graded and filtered rings}, J. Algebra \textbf{195} (1997), no.~2, 662--679. \MR{1469646}

\bibitem[Wei94]{Weibel}
Charles~A. Weibel, \emph{An introduction to homological algebra}, Cambridge Studies in Advanced Mathematics, vol.~38, Cambridge University Press, Cambridge, 1994. \MR{1269324}

\bibitem[Yek92]{yek92}
Amnon Yekutieli, \emph{Dualizing complexes over noncommutative graded algebras}, J. Algebra \textbf{153} (1992), no.~1, 41--84. \MR{1195406}

\bibitem[Yek20]{yekutielibook}
\bysame, \emph{Derived categories}, Cambridge Studies in Advanced Mathematics, vol. 183, Cambridge University Press, Cambridge, 2020. \MR{3971537}

\bibitem[YZ97]{Yekutieli_Zhang}
Amnon Yekutieli and James~J. Zhang, \emph{Serre duality for noncommutative projective schemes}, Proc. Amer. Math. Soc. \textbf{125} (1997), no.~3, 697--707. \MR{1372045}

\end{thebibliography}
\Addresses
\end{document}